\documentclass[11pt]{article}

\usepackage{epsfig}
\usepackage{latexsym}
\usepackage{amsmath}
\usepackage{amsfonts}
\usepackage{amssymb}
\usepackage{graphicx}%
\usepackage{algorithmic, algorithm}

\floatstyle{ruled}
\newfloat{algorithm}{tbp}{loa}

\floatname{algorithm}{\protect\algorithmname}
\providecommand{\algorithmname}{Algorithm}

\makeatletter

\newcommand{\Rmnum}[1]{\expandafter\@slowromancap\romannumeral #1@}
\makeatother

\makeatletter
\let\@fnsymbol\@arabic
\makeatother

\usepackage{accents,graphicx}

\makeatletter
\DeclareRobustCommand{\cev}[1]{%
  \mathpalette\do@cev{#1}%
}
\newcommand{\do@cev}[2]{%
  \fix@cev{#1}{+}%
  \reflectbox{$\m@th#1\vec{\reflectbox{$\fix@cev{#1}{-}\m@th#1#2\fix@cev{#1}{+}$}}$}%
  \fix@cev{#1}{-}%
}
\newcommand{\fix@cev}[2]{%
  \ifx#1\displaystyle
    \mkern#23mu
  \else
    \ifx#1\textstyle
      \mkern#23mu
    \else
      \ifx#1\scriptstyle
        \mkern#22mu
      \else
        \mkern#22mu
      \fi
    \fi
  \fi
}

\begin{document}

\pagestyle{myheadings} \markright{\sc On coloring numbers of graph powers \hfill} \thispagestyle{empty}

\newtheorem{theorem}{Theorem}[section]
\newtheorem{corollary}[theorem]{Corollary}
\newtheorem{definition}[theorem]{Definition}
\newtheorem{guess}[theorem]{Conjecture}
\newtheorem{claim}[theorem]{Claim}
\newtheorem{problem}[theorem]{Problem}
\newtheorem{question}[theorem]{Question}
\newtheorem{lemma}[theorem]{Lemma}
\newtheorem{proposition}[theorem]{Proposition}
\newtheorem{fact}[theorem]{Fact}
\newtheorem{acknowledgement}[theorem]{Acknowledgement}
\newtheorem{axiom}[theorem]{Axiom}
\newtheorem{case}[theorem]{Case}
\newtheorem{conclusion}[theorem]{Conclusion}
\newtheorem{condition}[theorem]{Condition}
\newtheorem{conjecture}[theorem]{Conjecture}
\newtheorem{criterion}[theorem]{Criterion}
\newtheorem{example}[theorem]{Example}
\newtheorem{exercise}[theorem]{Exercise}
\newtheorem{notation}[theorem]{Notation}
\newtheorem{observation}[theorem]{Observation}
\newtheorem{solution}[theorem]{Solution}
\newtheorem{summary}[theorem]{Summary}

\newtheorem{thm}[theorem]{Theorem} 
\newtheorem{prop}[theorem]{Proposition} 
\newtheorem{defn}[theorem]{Definition}

\newenvironment{proof}{\noindent {\bf
Proof.}}{\rule{3mm}{3mm}\par\medskip}
\newcommand{\remark}{\medskip\par\noindent {\bf Remark.~~}}
\newcommand{\pp}{{\it p.}}
\newcommand{\de}{\em}
\newcommand{\g}{\mathrm{g}}

\newcommand{\qf}{Q({\cal F},s)}
\newcommand{\qff}{Q({\cal F}',s)}
\newcommand{\qfff}{Q({\cal F}'',s)}
\newcommand{\f}{{\cal F}}
\newcommand{\ff}{{\cal F}'}
\newcommand{\fff}{{\cal F}''}
\newcommand{\fs}{{\cal F},s}
\newcommand{\cs}{\chi'_s(G)}

\newcommand{\G}{\Gamma}
\newcommand{\wrt}{with respect to }
\newcommand{\mad}{{\rm mad}}
\newcommand{\col}{{\rm col}}
\newcommand{\gcol}{{\rm gcol}}

\newcommand*{\ch}{{\rm ch}}
\newcommand*{\ra}{{\rm ran}}
\newcommand{\co}{{\rm col}}
\newcommand{\sco}{{\rm scol}}
\newcommand{\wc}{{\rm wcol}}
\newcommand{\dc}{{\rm dcol}}
\newcommand*{\ar}{{\rm arb}}
\newcommand*{\ma}{{\rm mad}}
\newcommand{\di}{{\rm dist}}
\newcommand{\tw}{{\rm tw}}
\newcommand{\scol}{{\rm scol}}
\newcommand{\wcol}{{\rm wcol}}
\newcommand{\td}{{\rm td}}
\newcommand{\edp}[2]{#1^{[\natural #2]}}
\newcommand{\epp}[2]{#1^{\natural #2}}
\newcommand*{\ind}{{\rm ind}}

\newcommand*{\QEDA}{\ensuremath{\blacksquare}} 
\newcommand*{\QEDB}{\hfill\ensuremath{\square}}  
\title{On coloring numbers of graph powers} 

\renewcommand{\thefootnote}{\arabic{footnote}} 
\author{
	H. A. Kierstead \footnotemark[1]$~^,$\footnotemark[4]~~~~
	Daqing Yang \footnotemark[2]$~^,$\footnotemark[5]~~~~
	Junjun Yi \footnotemark[3]$~^,$\footnotemark[6] 
}
\footnotetext[1]{School of Mathematical and Statistical Sciences, Arizona State University, Tempe, AZ 85287, USA}
\footnotetext[2]{Department of Mathematics,
	Zhejiang Normal University, Jinhua, Zhejiang 321004, China}
\footnotetext[3]{Center for Discrete Mathematics,
	Fuzhou University, Fuzhou, Fujian 350108, China.} 
\footnotetext[4]{E-mail: \small\texttt{kierstead@asu.edu}.}
\footnotetext[5]{Corresponding author,  grant numbers: 
	NSFC  11871439, 11471076. E-mail: \small\texttt{dyang@zjnu.edu.cn}.} 
\footnotetext[6]{E-mail: \small\texttt{1510163305@qq.com}.}

\maketitle

\begin{abstract}
	The weak $r$-coloring numbers  $\wc_r(G)$ 
	of a graph $G$ were introduced by the first two authors as a generalization of the usual coloring number $\co(G)$, and have since  found interesting theoretical and algorithmic applications. 
	This has motivated researchers to establish strong bounds on these parameters for various classes of graphs.

	Let $G^p$ denote the $p$-th power of $G$. We show that, all  integers $p >0$ and $\Delta \ge 3$ and graphs
	$G$ with $\Delta(G)\leq\Delta$ satisfy $\co(G^p) \in  O(p \cdot \wc_{\lceil p/2\rceil}(G)(\Delta-1)^{\lfloor p/2\rfloor})$; for fixed  tree width or fixed genus
	the ratio between this  upper bound and worst case lower bounds is polynomial in $p$.
	For the square of graphs $G$, we also show that, if the maximum average degree  $2k-2 < \mad(G) \leq 2k$,   
	then $ \col(G^2) \leq (2k-1)\Delta(G)+2k+1$.
\end{abstract}


\section{Introduction}

Let $G=(V,E)$ be a graph. For two vertices $x$ and $y$ in the same
component of $G$, the \emph{distance} $\di_{G}(x,y)$ between $x$
and $y$ is the length of a shortest $x,y$-path in $G$.
The \emph{$k$-th open neighborhood} $N_{G}^k(v)$ and \emph{$k$-th closed neighborhood} $N_{G}^k[v]$ of a vertex  $v\in V$ are defined by
$$N_{G}^k(v)=\{w\in V: \di_{G}(v,w)=k\}~\textrm{and}~N_{G}^k[v]=\{w\in V: \di_{G}(v,w)\le k\}.$$
As usual, we set $N_{G}(v)=N_{G}^{1}(v)$, $N_{G}[v]=N_{G}^{1}[v]$ and $d_{G}(v)=|N_{G}(v)|$.
Finally, we drop the subscripts $G$ in the above notations when $G$
is clear from the context.

The $p$-th \emph{power} of
$G$ is the graph $G^{p}=(V,E^{p})$, where $E^{p}=\{xy:1\leq\di_{G}(x,y)\leq p\}$.
Then $N_{G^{p}}(x)=N_{G}^{p}[x]-x$. 
Here we are concerned with the problem of bounding the chromatic number
and the list chromatic number of the $p$-th powers 
of graphs from various classes, particularly for fixed maximum degree $\Delta$ and arbitrary  $p$.  
Although more general, our results
improve on the known bounds for the chromatic number of graph powers of graphs excluding some fixed minor.

\subsection{Generalized coloring numbers}

For a graph $G=(V,E)$, let $\Pi:=\Pi(G)$ be the set of total
orderings of the vertex set $V$. For $\sigma\in\Pi$ and $x\in V$,
set
\begin{enumerate}
	\item $V_{\sigma}^{l}(x)=\{y\in V: y<_{\sigma}x\}$, $V_{\sigma}^{l}[x]=V_{\sigma}^{l}(x)+x$;
	and
	\item $V_{\sigma}^{r}(x)=\{y\in V: x<_{\sigma}y \}$, $V_{\sigma}^{r}[x]=V_{\sigma}^{r}(x)+x$.
\end{enumerate}
Thus $\{V_{\sigma}^{l}(x),\{x\},V_{\sigma}^{r}(x)\}$ partitions $V$
into the \emph{left set} of $x$, singleton $x$, and the \emph{right
	set} of $x$. The coloring number of $G$, denoted $\co(G)$, is defined
by
\[
\co(G)=\min_{\sigma\in\Pi}\max_{x\in V}|N[x]\cap V_{\sigma}^{l}[x]|.
\]
Greedily coloring the vertices of $G$ in an order
that witnesses its coloring number shows that
\[
\chi(G)\leq\chi_{l}(G)\leq\co(G),
\]
where $\chi(G)$ is the chromatic number of $G$, and $\chi_{l}(G)$ is the  list
chromatic number of $G$.

\emph{Generalized coloring numbers} were first introduced by Kierstead and Yang in \cite{KY} after 
similar notions were explored by various authors \cite{MR1198403,K,KT,KT-OGC}
in the cases $k=2,4$. Let $k\in\mathbb{Z}^{+}\cup\{\infty\}$. A
vertex $y$ is \emph{weakly k-reachable} from $x$ with respect to
$\sigma$ if $y\in V_{\sigma}^{l}[x]$ and there is an $x,y$-path
$P$ with $\|P\|\le k$ and $V(P)\subseteq V_{\sigma}^{r}[y]$.
Let $W_{\sigma}^{k}[x]$ be the set of vertices that are weakly $k$-reachable
from $x$ with respect to $\sigma$. 
The \emph{weak $k$-coloring number},  denoted $\wc_{k}(G)$, of $G$ 
is defined by:
\[
\wc_{k}(G)=\min_{\sigma\in\Pi}\max_{x\in V}|W_{\sigma}^{k}[x]|. 
\]
Observe that $\co(G)=\wc_{1}(G)$. 

The weak coloring numbers have found many important and diverse applications (cf. \cite{D-2013,GKR,HKQ}). 
As shown by Ne{\v s}et{\v r}il  and Ossona de Mendez \cite[Lemma 6.5]{NO-Sparsity}, 
they also provide a gradation
between the coloring number and the tree-depth $\td(G)$ of a graph  $G$ 
as follows:
$$\co(G)=\wc_1(G)\le\wc_2(G)\le\dots\le\wc_{\infty}(G)=\td(G).$$

Graph classes with \emph{bounded
	expansion} (a notion extending graph classes excluding a minor
or topological minor)
 were first introduced by Ne{\v s}et{\v r}il and Ossona de Mendez \cite{MR2397335,NO-Sparsity}.
Zhu \cite{Zhu-BGCN} (also see \cite{MR2519201}) characterized graph classes with {bounded
	expansion} as those classes $\mathcal{C}$ for which there is a function
$f:\mathbb{Z}^{+}\rightarrow\mathbb{Z}^{+}$ such that all graphs
$G\in\mathcal{C}$ and all integers $k\in\mathbb{N}$ satisfy $\wcol_{k}(G)\leq f(k)$. 

The following theorem gives upper bounds on the 
weak $k$-coloring numbers for various graph classes. Items (2--4) below are essentially due to \cite{HOQRS}, where all are proved using the same technique; here by using an observation in \cite{HKQ},  their results are improved a bit by adding the last negative term. Item 5 
is  Proposition 28  of \cite{HW}.

\begin{thm} \label{wc-g-p}
	All positive integers $k$ and graphs $G$ satisfy:
	\begin{enumerate}
		\item  \cite{GKRSS} ~ $\wc_{k}(G)\leq\binom{k+t}{t}$, if $\tw(G)\leq t\leq\Delta+1$, where  $\tw(G)$ is the tree-width of $G$,  and this is sharp;
		\item \cite{HKQ,HOQRS} $\wc_{k}(G)\leq\binom{k+t-2}{t-2}(t-3)(2k+1)-k(t-3)$, if $t\geq4$
		and $G$ has no $K_{t}$ minor;
		\item \cite{HKQ,HOQRS} $\wc_{k}(G)\leq(2g+\binom{k+2}{2})(2k+1)-k$, if $G$  has genus $g$;
		\item \cite{HKQ,HOQRS}  $\wc_{k}(G)\leq\binom{k+2}{2}(2k+1)-k$, if $G$ is planar; 
		
		\item  \cite{HW} 	$\wc_{k}(G) \leq s(t-1)\binom{k+s}{s}(2k+1)$, if $G$ is  $K^{\ast}_{s,t}$-minor-free, where $K^{\ast}_{s,t}$ is the complete join of $K_s$ and $\overline{K_t}$.  
	\end{enumerate}
\end{thm}

\subsection{Parameters for measuring density}

The coloring number is closely related to various parameters for measuring
the local density of a graph.  
The \emph{arboricity} of a graph $G$, denoted $\ar(G)$, is the minimum
number of forests required to cover the edges of $G$. By 
Nash-Williams' Theorem \cite{N-W}, 
$\ar(G)=\max_{H\subseteq G,|H|\geq2}\left\lceil \frac{\left\Vert H\right\Vert }{|H|-1}\right\rceil$.
The \emph{maximum average degree }of $G$ is
$\ma(G)=\max_{\emptyset\ne H\subseteq G}\frac{2\left\Vert H\right\Vert }{|H|}$. 
The following proposition is well known and easy to prove.

\begin{prop} Every graph $G$ satisfies
	$$\chi(G)\leq\chi_{l}(G)\leq\co(G)  
	\le \lfloor\ma(G)\rfloor+1\leq2\ar(G). $$ 
\end{prop}

For a graph $G=(V,E)$, 
let $\vec{E}=\{\vec{e},\cev e:e\in E\}$ be the set of orientations 
of its edges. Define a \emph{weak orientation} of $G$ to be a function $w:\vec{E}\rightarrow\mathbb{R}$
such that $w(\vec{uv})+w(\cev{uv})=1$ and $w(\vec{uv}),w(\cev{uv})\geq0$ for all
$uv\in E$. We say that $G$ is weakly oriented if it has been assigned a weak orientation. Observe that 
 ordinary (unoriented) graphs can be interpreted as weakly oriented graphs whose edges have weight $1/2$ in both directions, and  oriented graphs can be interpreted as weakly oriented graphs whose weights are $0,1$-valued.  
 Define the \emph{out-weight} $w^{+}(u)$ of $u$
and the \emph{maximum out-weight} $\Delta_{w}^{+}(G)$ of $G$ by 
\[
w_{G}^{+}(u):=w^{+}(u):=\sum_{uv\in E}w(\vec{{uv}})\quad\text{and}\quad\Delta_{w}^{+}(G):=\max_{u\in V}w_{G}^{+}(u).
\]
Using standard notation, let $\Delta^{+}(\vec G)$ denote the maximum outdegree of an oriented graph $\vec G$. 

\begin{prop}\label{fact}
Every graph $G$ satisfies both:
\begin{enumerate}
\item \label{F1} $2\min_{w}\Delta_{w}^{+}(G)=\mad(G)$, where $w$ runs over all weak orientations of $G$ and 
\item \label{F2} (cf. Hakimi \cite{hak}) $2\min_{\vec G} \Delta^{+}(\vec G)=\lceil\operatorname{mad}(G)\rceil$.
\end{enumerate} 
\end{prop}

\begin{proof}
First we prove item \ref{F1}. 
For any subgraph $H\subseteq G$, and weak orientation $w$, 
\[
\left\Vert H\right\Vert =\sum_{e\in E(H)}(w(\vec{e})+w(\cev e))=\sum_{v\in V(H)}w_{H}^{+}(v)\leq |H|\cdot\Delta_{w}^{+}(H) \leq |H|\cdot\Delta_{w}^{+}(G).
\]
Setting $m=\mad(G)=\max_{H\subseteq G}\left\Vert H\right\Vert /|H|$, we have $m\le 2\Delta_{w}^{+}(G)$. 

Now we find a weak orientation $w$ with $2\Delta_{w}^{+}(G)\leq m$. 
Fix $H$ witnessing $m$, 
  and let  $n=|H|$.
Pick $w:\vec{E}\rightarrow\{k/n:k=0,\dots,n\}$ 
so that  the
 \emph{excess weight} $$b(w):=\sum\{2w^{+}(x)-m:x\in V \text{ and }2w^{+}(x)>m\}$$ is
minimum. This is possible since there are only $(n+1)^{\left\Vert G\right\Vert }$
choices for $w$. It suffices to show $b(w)=0$. Suppose not. Then there is a vertex
$x$ with $2w^{+}(x)>m$. By the choice of $w$, $2w^{+}(x)-m\geq1/n$. 
Let $S$
be the set of vertices $v\in V$ for which there is a path $P_{v}=v_{0}\dots v_{s}$
with $x=v_{0}$ and $v=v_{s}$ such that every forward oriented edge has positive
weight (and so weight at least $1/n$). Set $H'=G[S]$. 
If $v\in S$ and $w\in N(v)\smallsetminus S$
then $w(\vec{vw})=0$, so $w_{H'}^{+}(v)=w_{G}^{+}(v)$. Since $2\left\Vert H'\right\Vert /|H'|\leq m$,
there is $v\in S$ with $2w_{G}^{+}(v)=2w_{H'}^{+}(v)<m$, and so $2w^{+}(v)\leq m-1/n$.
 Define a new weak orientation $w'$ by decreasing (increasing) the weight of each
forward (backward) edge of $P_{v}$ by $1/n$. Now $b(w')<b(w)$, a contradiction.

For the proof of item \ref{F2}, replace ``weak orientation'' with ``$0,1$-orientation'', set $m=\lceil \mad(G) \rceil= \max_{H\subseteq G}\lceil\left\Vert H\right\Vert /|H|\rceil $, and set  $n=1$ in the proof of item \ref{F1}.  
\end{proof}


In Section 2, for general $p$, we study the coloring number of the $p$-th power of graphs $G$.
 We show that, all positive integers $p$ and $\Delta$ and graphs
$G$ with $\Delta(G)\leq\Delta$ satisfy $\co(G^p)$ 
$\le O(p \cdot \wc_{\lceil p/2\rceil}(G)(\Delta-1)^{\lfloor p/2\rfloor})$; 
for fixed  tree width or fixed genus
the ratio between this  upper bound and worst case lower bounds is polynomial in $p$.
In Section 3, we study the coloring number of the square of graphs $G$; 
we show that, if the maximum average degree  $2k-2 < \mad(G) \leq 2k$,    
then $ \col(G^2) \leq (2k-1)\Delta(G)+2k+1$.


\section{Coloring numbers of  graph powers}

\subsection{Previous results}

If $G$ is a connected graph with diameter at most $p$ then $G^{p}=K_{|G|}$.
As observed in \cite{AH}, if $T$ is a maximum tree of height 
$\lfloor p/2\rfloor$  and $\Delta(T)\leq\Delta$ then
\begin{equation} \label{LB}
\chi(T^{p})\geq1+\Delta\sum_{i=0}^{\lfloor p/2\rfloor-1}(\Delta-1)^{i}=1+\Delta\frac{(\Delta-1)^{\lfloor p/2\rfloor}-1}{\Delta-2}= 
\frac{\Delta(\Delta-1)^{\lfloor p/2\rfloor}-2}{\Delta-2} =: L. 
\end{equation}
For the square of planar graphs,   Agnarsson and Halld{\'o}rsson  \cite{AH}  proved that if $G$ is a planar graph with $\Delta = \Delta(G) \ge 750$, then $\co(G^2) \le \lceil \frac{9}{5}\Delta \rceil$; and this is sharp.
An upper bound on the coloring number of $G^p$ is provided by the following
theorem. 
\begin{thm}[Agnarsson and Halld{\'o}rsson \cite{AH}]
	\label{AH-thm} For all $p,\Delta\in \mathbb Z^+$ and graphs
	$G$ with $\Delta(G)\leq\Delta$,
	\[
	\ar(G^{p})\le2^{p+1}(\ar(G))^{\lceil p/2\rceil}\Delta^{\lfloor p/2\rfloor}.
	\]
\end{thm}

This upper bound was improved for chordal graphs in \cite{Kral-PowChordalG}.
\begin{thm}[  Kr{\'a}l' \cite{Kral-PowChordalG}]
	\label{Kral-thm} For all $p,\Delta\geq2\in\mathbb Z^+$ and
	chordal graphs $G$ with $\Delta(G)\leq\Delta$,
	\[
	\co(G^p)\leq \left\lfloor \sqrt{\frac{91p-118}{384}}(\Delta+1)^{(p+1)/2} \right\rfloor +\Delta + 1.
	\]
\end{thm}

\subsection{New result}

In this subsection we improve the known bounds on $\co(G^p)$ for graph
classes, including planar and chordal graphs, whose weak coloring
numbers grow subexponentially.

\begin{thm} \label{main-thm} 
	All  integers $p >0$ and $\Delta \ge 3$ and graphs
	$G=(V,E)$ with $\Delta(G)\leq\Delta$ satisfy
	\[
	\ma(G^{p})
	\leq
	\frac{\Delta}{\Delta-2}2\lfloor\frac{p+1}{2}\rfloor  \wc_{\lceil p/2\rceil}(G)(\Delta-1)^{\lfloor p/2\rfloor}.
	\]
\end{thm}

\begin{proof}
	Suppose $G$ is a graph with $\Delta(G)\leq\Delta$ and $\sigma\in \Pi(G)$ witnesses that $\wc_{\lceil p/2\rceil}(G)=q$. 

	By Proposition~\ref{fact}.\ref{F1}, 
	 it suffices to construct a weak orientation $w$ such that 
	\[w^{+}(G)\le  \frac{\Delta}
	{\Delta-2}\lfloor\frac{p+1}{2}\rfloor q(\Delta-1)^{\lfloor p/2 \rfloor}.\]
	
	Consider any edge $e=uv$ 
	in $G^{p}$.  
	Then there is a path of length
	at most $p$ that connects $u$ and $v$. 
	Choose  	
	such a $u,v$-path $Q_{e}=u_{0}\dots u_{s}\subseteq G$ with minimum length $s=\|Q_{e}\|\le p$. 
	Let 
	$l_{e}$ be the $\sigma$-least vertex in $Q_{e}$. If $e$ has a unique end, say $u$, with the distance $\|uQ_{e}l_{e}\|\ge  s/2$, then set $w(\vec {uv})=1$ and $w(\vec {vu})=0$; else set $w(\vec {uv})=\frac{1}{2}=w(\vec {vu})$.

	Consider any vertex $u\in V$, and suppose $e=uv\in E$ with $w(\vec{uv})>0$. 
	Then $Q_e$ has the form $u_0Q_{e}u_hQ_{e}u_iQ_{e}u_s$, where $l_{e}=u_{i}$, $h=0$ if $i\le \lceil s/2\rceil$ and $h=i- \lceil s/2\rceil$ else. 
	Then
	\begin{align}\label{path}
	\mbox{$0\le h\le \lfloor \frac{p}{2}\rfloor$, $u_h\in N^h(u)$, $l_{e} = u_i \in W^{ \lceil s/2\rceil}_{\sigma}[u_h]$ and $v\in N^{s-i}(l_{e})$.}  
	\end{align}
	Moreover, if $h=0$ then $w(\vec{uv})=\frac{1}{2}=w(\vec{vu})$.
	
	 Thus  $w^{+}(u)$ is at most the number of ways to pick $h>0,u_i,v$ satisfying \eqref{path} plus one half the number ways to pick $h=0,u_{i},v$ satisfying \eqref{path}. By the definition of the $h$-th open neighborhood,  
	$|N^h(u)| \le 
	\Delta(\Delta-1)^{h-1}$ and $|N^{s-i}(l)|\le (\Delta-1)^{s-i}$; also $|W^{ \lceil p/2\rceil}_{\sigma}[u_h]|\le q$. Noticing that the special case $h=0$ accounts for the first term on the RHS of  \eqref{h=0}, 
	we have 
	\begin{align}
	w^{+}(u) \label{h=0}
	&\le \frac{1}{2}q\sum_{j=0}^{\lfloor p/2\rfloor}(\Delta-1)^{j}+\sum_{h=1}^{\lfloor{p}/{2}\rfloor} \Delta(\Delta-1)^{h-1}q\sum_{j=0}^{\lfloor p/2\rfloor-h}(\Delta-1)^{j}
	\\&\le \frac{1}{2}q\frac{\Delta(\Delta-1)^{\lfloor p/2\rfloor}}{\Delta-2}+\sum_{h=1}^{\lfloor{p}/{2}\rfloor} \Delta(\Delta-1)^{h-1}q\frac{(\Delta-1)^{\lfloor p/2\rfloor-h+1}}{\Delta-2}\notag
	\\&\le \frac{\Delta}{\Delta-2}\lfloor\frac{p+1}{2}\rfloor q(\Delta-1)^{\lfloor p/2\rfloor}. \notag
	\end{align}
	\end{proof} 

The ratio obtained by dividing the bound of Theorem \ref{AH-thm} 
by the lower bound $L$ from eq. \eqref{LB},  
is clearly exponential. Dividing the bound of Theorem \ref{Kral-thm} 
by $L$ we get: 
\[
\begin{split}
 & \frac{\left\lfloor \sqrt{\frac{91p-118}{384}}(\Delta+1)^{(p+1)/2} \right\rfloor +\Delta + 1}{\frac{\Delta(\Delta-1)^{\lfloor p/2\rfloor}-2}{\Delta-2}} 
  \ge \frac{\sqrt{\frac{91p-118}{384}}(\Delta+1)^{(p+1)/2}}{(\Delta-1)^{\lfloor p/2\rfloor}} \cdot \frac{(\Delta-1)^{\lfloor p/2\rfloor}}{\frac{\Delta(\Delta-1)^{\lfloor p/2\rfloor}-2}{\Delta-2}} \\
 & \ge \sqrt{\frac{91p-118}{384}} 
 \cdot \frac{\Delta-2}{\Delta}
 \cdot \left(1+ \frac{2}{\Delta -1}\right)^{p/2}, 
\end{split}
\]
which is also exponential. But the ratio obtained by dividing the bound of Theorem \ref{main-thm} 
by $L$ is polynomial in $p$  whenever $\wc_p(G)$ is polynomial in $p$.  In particular this is the case for graphs with bounded tree width and graphs with no $K_t$ minor, including graphs with bounded genus; and graphs with no $K^{\ast}_{s,t}$-minor (where $K^{\ast}_{s,t}$ is the complete join of $K_s$ and $\overline{K_t}$).

\section {On the coloring number of the square of graphs}

\subsection{Previous results} 

The study of $\chi(G^2)$ was initiated by Wegner in \cite{W}, and has been actively studied ever since.
In \cite{CC}, Charpentier made the following conjectures.  

\begin{conjecture} \label{conj-CC-1} (\cite{CC})
	There exists an integer $D$ such that every graph $G$ with $\Delta(G) \ge D$ and $\mad(G) < 4$ has $\chi(G^2) \le 2 \Delta(G)$.  
\end{conjecture}

\begin{conjecture} \label{conj-CC-2} (\cite{CC})
	For each integer $k \ge 3$, there exists an integer $D_k$ such that every graph $G$ with $\Delta(G) \ge D_k$ and $\mad(G) < 2k$ has 
	$\chi(G^2) \le k\Delta(G) - k$.
\end{conjecture}

In \cite{CC},  
some examples are given to show 
that Conjectures \ref{conj-CC-1} and \ref{conj-CC-2} 
are best possible if they are true.  
In \cite{KP}, Kim and Park disproved Conjectures \ref{conj-CC-1} and \ref{conj-CC-2} 
by showing that, for any positive integer $D$, there is a graph $G$ with $\Delta(G) \ge D$ and $\mad(G) < 4$ such that $\chi(G^2) = 2\Delta(G)  + 2$; 
for any integers $k$ and $D$ with $k \ge 3$ and $D \ge k^2 - k$, there exists a graph $G$ with $\mad(G) < 2k$ and $\Delta(G) \ge D$, such that $\chi(G^2) \ge  k\Delta(G) + k$. 

For the upper bounds, the following result is \cite[Theorem 4]{HKP} by Hocquard, Kim and Pierron (very recently); a similar version was given by Charpentier in \cite{CC}. The version in \cite{HKP} is proved by using a variant of discharging, and  fixed some errors and inaccuracies of the original proof. 

\begin{theorem} \label{thm-CC-HKP} (\cite{CC,HKP}) 
	Let $k$ be an integer and $G$ be a graph with $\mad(G) < 2k$. Then
\[
\chi(G^2) \le \max \{ (2k-1)\Delta(G)-k^2+k+1, (2k-2)\Delta(G)+2k^3+k^2+2,  (k-1)\Delta(G)+k^4+2k^3+2 \}. 
\]
\end{theorem}

Kim and Park \cite{KP} proved the following theorem. 
\begin{theorem} \label{thm-KP} (\cite{KP})
	Let $c$ be an integer such that $c \ge 2$. If a graph $G$ satisfies $\mad(G) < 4-\frac{1}{c}$ and $\Delta(G) \ge 14c - 7$, then $\chi_l(G^2) \le 2\Delta(G)$.  
\end{theorem}

  Bonamy, L\'{e}v\^{e}que, and Pinlou in \cite{BLP3} proposed the following question.

\begin{question} \label{Que-1} (\cite{BLP3}) 
	What is, for any $C \ge  1$, the maximum $m(C)$ such that any graph $G$ with $\mad(G) < m(C)$ satisfies   $\chi_l(G^2) \le \Delta(G) + C$. 
\end{question}

As a natural generalization of Question \ref{Que-1}, the following question seems interesting, especially by taking Conjectures \ref{conj-CC-1}, \ref{conj-CC-2} and their recent developments into considerations.  

\begin{question} \label{Que-2}  
	What is, for a given integer $k \ge 1$ and any $C$ (if $k=1$, then  $C \ge  1$), the minimum $m(C)$ such that any graph $G$ with $\mad(G) \le  2k - m(C)$ satisfies  
	  $\chi_l(G^2) \le k\Delta(G) + C$. 

\end{question}

\subsection{New result }
\label{subsecVG}

The techniques we use in this section have their roots in the study of coloring games on graphs, in particular, the Harmonious Strategy introduced in \cite{KY2}.  
 In fact, a more recent game theoretic result of Yang  \cite[Theorem 4.5]{Y-q}, 
 already yields the following corollary.

\begin{corollary} (\cite{Y-q}) \label{col1} 
	Let $k$ be a positive integer, and let $G$ be a graph with $\mad(G)\leq2k$ and $\Delta(G) \geq 2k-2$. Then $\col(G^2)\leq (3k-2)\Delta(G)-k^2+4k+2$.
\end{corollary}

The next theorem is our result on the coloring number of the square of graphs. In its proof we construct an ordering of the vertices of a graph $G$ to witness the given bound on $\col(G^{2})$. 
 This is done by iteratively adding new vertices to the end of the initial segment of already ordered vertices.  (Contrast this with the usual method of adding new vertices at the front of the final segment  already constructed.)   There is a natural tension between adding a vertex too late and thus giving it too many earlier neighbors, and adding it too soon, and thus giving too many other vertices an earlier neighbor. The Harmonious Strategy provides a scheme for balancing these considerations by ensuring that no vertex is chosen before its out-neighbors and  distance-$2$ out-neighbors have been considered. See \cite{VdHK} for another application of the Harmonious Strategy to a non-game problem.

\begin{theorem} \label{col-main}
	Let $k >0$ be an integer. If $G$ is a graph with $2k-2 < \mad(G) \leq 2k$,    
	then 
	\begin{align}
	 \col(G^2) \leq (2k-1)\Delta(G)+2k+1.\label{*}
	 \end{align}
\end{theorem}

\begin{proof}
Suppose $G=(V,E)$ is a graph with $2k-2 <\mad(G)\leq2k$.     
Thus 
\begin{align}
2k-1\le \Delta:=\Delta(G).\label{D}
\end{align}
By Proposition \ref{fact}.\ref{F2},  
 $G$ has  an orientation $\vec{G}=(V,\vec{E})$ with $\Delta^+:=\Delta^+(\vec{G})\leq k$. Let $L\in \Pi(G)$.

Given a path $P=x_{0}\dots x_{l}\subseteq G$, define the {\em sign-sequence} of $P=x_{0}\dots x_{l}$  to be the sequence $s(P)$ with length $l$ whose $i$-th symbol is ``$+$'' if $v_{i-1}v_{i}\in \vec{E}$ and ``$-$'' if  $v_{i}v_{i-1}\in \vec{E}$. 
For any $x\in V$ and sign-sequence $s\in\{+,-,++,+-,-+,--\}$, let $N^s(x)$ denote the set of vertices $y$ such that there is a shortest $x,y$-path $P\subseteq G$ with $s(P)=s$. Put $d^{s}(x)=|N^{s}(x)|$. 
Also put $ N^{+,++}(x)=N^{+}(x)\cup N^{++}(x)$, 
$d^{+,++}(x)=d^{+}(x)+ d^{++}(x)$, 
$ N^{-,--}(x)=N^{-}(x)\cup N^{--}(x)$ and  
$d^{-,--}(x)=d^{-}(x)+ d^{--}(x)$.

Our task is to construct $\sigma\in\Pi(G)$ witnessing \eqref{*}.  
To do this, we design an algorithm that {\em collects} vertices one at a time. Each time a vertex is collected, it is deleted from the set $U$ of uncollected vertices, and put at the end of the initial segment of $\sigma$ that has  already been constructed. We maintain a set $S_{x} \subseteq  N^{+,++}(x)$ for each  $x\in V$.
The vertex sets $U$ and $S_{x}$ are dynamic---they are updated as the algorithm runs.

We start without any collected vertices, so $U:=V$.  
For all $x\in V$, set $S_{x}:=N^{+,++}(x)$.  
Then  we 
 run Algorithm~\ref{A1} (see below).

\begin{algorithm}
\caption{}
\label{A1}
\begin{algorithmic}[1]
\WHILE {$U\ne \emptyset$}
\STATE $x:=L$-$\min U$
\WHILE {$S_{x} \cap U \neq \emptyset$}
\STATE $y:=L$-$\min S_{x} \cap U$
\STATE $S_{x}:=S_{x}-\left\{ y \right\} $
\IF {$S_{x}\cap U=\emptyset$}
\STATE collect $x$
\ENDIF
\STATE $x:=y$
\ENDWHILE
\STATE collect $x$
\ENDWHILE
\end{algorithmic}
\end{algorithm}

\medskip

When vertex $u\in U$ with $S_{u} \cap U \neq \emptyset$ is assigned to variable $x$ at Line~2 or at Line~9   
	 and then $v\in S_u\cap U$ 
	 is immediately assigned to variable  $y$ at Line~4, we say that {\em $u$ contributes to $v$} and {\em $v$ receives a contribution from $u$}.

	When a vertex $w$ assigned to variable $x$ is collected at Line 7 or Line 11, we have $S_{w} \cap U = \emptyset$, so every vertex $v\in N^{+,++}(w)$ has received a contribution from $w$ or has been collected. Thus:
	\begin{align}\label{C}
	\mbox{When $w\in V$ is collected, it has contributed to all
	$y\in N^{+,++}(w)\cap U$.
	}
	\end{align}
	
	When a vertex $v\in V$ receives a contribution at Line~4, it is still uncollected. It is immediately 
	assigned to variable $x$ at Line~9. If $S_{v}\cap U=\emptyset$ then the inner while-loop ends, 
	and  $v$ is collected at Line~11; else $v$ contributes to some $u\in S_{v}\cap U$ at Line~4, and $|S_{v}\cap U|$ is reduced by $1$ at Line~5. If now $S_{v}\cap U=\emptyset$ then $v$ is immediately collected at Line~7. Thus:
	\begin{align}\label{NC}
	\mbox{Each $v\in V$  receives at most $d^{+,++}(v)$ contributions.}
	\end{align}

Consider any uncollected $v\in U$. It suffices to prove that $v$ has at  most $(2k-1)\Delta+2k$ collected neighbors in $G^{2}$, i.e., 
  \[|N^{2}[v]\setminus U|\le(2k-1)\Delta+2k.\]
	For all $w\in N^{-,--}(v)\setminus U$, we have $v\in N^{+,++}(w)$ and $w$ is collected before $v$. By \eqref{C}, $w$ has contributed to $v$. By \eqref{NC}, $v$ has received at most $d^{+,++}(v)$ contributions. Thus:
	 	\begin{align}
	|N^{-,--}(v)\setminus U| \le d^{+,++}(v).\label{N}
	\end{align}
	Now we have:
	\begin{align*}
	|N^{2}[v]\setminus U| &\le|N^{+,++}(v) \cup (N^{-,--}(v)\setminus U) \cup N^{+-}(v)\cup N^{-+}(v) | \\
	&   \leq  d^{+,++}(v) + |N^{+}(v)\cup N^{++}(v)\cup N^{+-}(v)| +d^{-+}(v) &\text{(by \eqref{N})}\\
	& \leq (k^2 + k)+d^{+}(v)\Delta+(\Delta-d^{+}(v))(k-1)\\ 
	& = (k^2 + k)+d^{+}(v)(\Delta-k+1) + \Delta(k-1) \\ 
	& \le (k^2 + k)+k(\Delta-k+1) + \Delta(k-1)  &\text{(by \eqref{D})} \\ 
	&=(2k-1)\Delta+2k.
	\end{align*}
\end{proof}

\noindent {\bf Acknowledgement:} We thank two anonymous referees for their  suggestions and comments that helped improve the presentation of the paper.


\end{document}